\documentclass[11pt]{amsart}
\usepackage{mathrsfs,latexsym,amsfonts,amssymb}
\newtheorem{theorem}{Theorem}[section]
\newtheorem{lemma}[theorem]{Lemma}
\newtheorem{corollary}[theorem]{Corollary}
\newtheorem{question}[theorem]{Question}

\newtheorem{proposition}[theorem]{Proposition}
\theoremstyle{definition}

\theoremstyle{remark}

\begin{document}
\title[Strongly topologically orderable gyrogroups with a suitable set]
{Strongly topologically orderable gyrogroups with a suitable set}

\author{Jiamin He}
\address{(Jiamin He): School of mathematics and statistics, Minnan Normal University, Zhangzhou 363000, P. R. China}
\email{hjm1492539828@163.com}

\author{Jiajia Yang}
\address{(Jiajia Yang): School of mathematics and statistics, Minnan Normal University, Zhangzhou 363000, P. R. China}
\email{3561333253@qq.com}

\author{Fucai Lin*}
\address{(Fucai Lin): 1. School of mathematics and statistics, Minnan Normal University, Zhangzhou 363000, P. R. China; 2. Fujian Key Laboratory of Granular Computing and Application, Minnan Normal University, Zhangzhou 363000, P. R. China}
\email{linfucai2008@aliyun.com; linfucai@mnnu.edu.cn}

\thanks{The authors are supported by Fujian Provincial Natural Science Foundation of China (No: 2024J02022) and the NSFC (Nos. 11571158).}
\thanks{* Corresponding author}


\begin{abstract}
 A  discrete subset $S$ of a topologically gyrogroup $G$ is called a {\it suitable set} for $G$ if $S\cup \{1\}$ is closed and the subgyrogroup generated by $S$ is dense in $G$, where $1$ is the identity element of $G$. In this paper, we mainly study the existence of suitable set of strongly topologically orderable gyrogroups, which extends some result in some papers in the literature. In particular, the existences of suitable set of each
 locally compact or not totally disconnected strongly topologically orderable
gyrogroup are affirmative.
\end{abstract}

\maketitle

\section{Introduction}
The concept of suitable sets for topological groups was introduced by Hofman and Morris \cite{key14} in 1990. Until now, numerous significant results on suitable sets for topological groups have been obtained by many topology scholars. However, the topic to study the existence of suitable set for topological groups is far from over. For example, M. Tkachenko in \cite{key16} posed the following open question, which is still unknown for us.

\begin{question}\cite[Problem 1.5]{key16}\label{q0}
Does every topologically oederable group contain a (closed) suitable set?
\end{question}

Indeed, M. Venkataraman, M. Rajagopalan and T. Soundararajan in \cite{key8}  proved that each not totally disconnected topologically orderable group $G$ contains an open normal subgroup which is topologically isomorphic to the additive group $\mathbb{R}$ of real numbers endowed with its usual topology; thus $G$ is metrizable. Hence we have the following theorem.

\begin{theorem}\label{t1}
Each not totally disconnected topologically orderable group has a suitable set.
\end{theorem}

 In 2017,  W. Atiponrat \cite{key1} introduced the concept of topological gyrogroups and discussed some topological properties of topological gyrogroups. Then, in 2019, M. Bao and F. Lin in \cite{key12} introduced strongly topological gyrogroups.
In this paper, we mainly discuss the existence of suitable set in the class of strongly topologically orderable gyrogroups, which gives a generalizations of Theorem~\ref{t1}.

This paper is organized as follows.

In Section 2, we introduce
necessary notations and terminology which are used in
the paper.

In Section 3, we mainly prove that every strongly topologically orderable gyrogroup is either metrizable or has a totally ordered local base $\mathcal{H}$ at the identity element, consisting of clopen $L$-gyrosubgroups, such that $gyr[x, y](H)=H$ for any $x, y\in G$ and $H\in \mathcal{H}$; moreover, we prove that every strongly topologically orderable gyrogroup is hereditarily paracompact.

In Section 4, we show that every locally compact, totally disconnected, strongly topologically orderable gyrogroup contains a suitable set.

In Section 5, we prove that if a strongly topologically orderable gyrogroup has a (closed) suitable set, then its dense subgyrogroup also has a (closed) suitable set.

\section{Preliminaries}
In this section, we introduce the
necessary notation and terminology which are used in
the paper. Denote the sets of real numbers, positive integers and all non-negative integers by $\mathbb{R}$, $\mathbb{N}$, and $\omega$, respectively. Readers may refer to
\cite{AT, Eng} for terminologies and notations not
explicitly given here.

A \emph{groupoid} (or \emph{magma}) \cite{key1} is an algebraic structure $(G,\oplus)$ which consists of a non-empty set $G$ and a binary operation $\oplus:G\times G\to G$.
Then we say that a groupoid $(G,\oplus)$ is called a \emph{gyrogroup} when its binary operation satisfies the following conditions:
	\begin{enumerate}
		\item[(1)] There is exact one identity element $1\in G$ such that $1\oplus g=g\oplus 1$ for each $g\in G$.
		
\smallskip
        \item[(2)] For each $g\in G$, there is exactly one inverse element $\ominus g\in G$ satisfying $\ominus g\oplus g=1=g\oplus(\ominus g$).
		
\smallskip
		\item[(3)] For any $g,h\in G$, there exisits a $gyr[g,h]\in Aut(G,\oplus)$ satisfying $g\oplus(h\oplus f)=(g\oplus h)\oplus gyr[g,h](f)$ for each $f\in G$.
		
\smallskip
		\item[(4)] For any $g,h\in G, gyr[g,h]=gyr[g\oplus h,h]$.
	\end{enumerate}

Then we can define the concept of subgyrogroup in a gyrogroup, see \cite{key3}. Clearly,  a nonempty subset $H$ of a gyrogroup $G$ is a subgyrogroup if and only if for any $a, b\in H$, we have $\ominus a\in H$ and $a\oplus b \in H$. A subgyrogroup $H$ of $G$ is called an \emph{$L$-subgyrogroup} \cite{key3}, if $gyr[g,h](H)=H$ for all $g\in G$ and $h\in H$.

The following theorem shows that the left cosets of each $L$-subgyrogroup of a gyrogroup $G$ forms a disjoint partition.

\begin{theorem} (\cite[Theorem 20]{key3})\label{ttttt}
	Let $G$ be a gyrogroup and $H$ an $L$-subgyrogroup of $G$. Then the family of left cosets
	\[
	\{a \oplus H:a\in G\}
	\]
	forms a disjoint partition of $G$.
\end{theorem}

Let $\mathcal{T}$ be a topology on a gyrogroup $G$. We say that $G$ is a \emph {topological gyrogroup} \cite{key1}  if the binary operation $\oplus \colon G \times G \to G$ and the inversion operation $\ominus \colon G \to G$ are continuous, where $G \times G$ is endowed with the product topology.
In particular, a topological gyrogroup $G$ is called a \emph{strongly topological gyrogroup} \cite{key10} if it admits a neighborhood base $\mathcal{U}$ of the identity element $1$ such that $gyr[x,y](U)=U$ for any $x,y\in G$ and $U\in \mathcal{U}$. For convenience, we say that $G$ is a {\it strongly topological gyrogroup with neighborhood base $\mathcal{U}$ of 1}.

A (strongly) topological gyrogroup $(G,\mathcal{T})$ is called to be ({\it strongly}) {\it topologically orderable gyrogroup} if there is a total order $\le$ on $G$ such that the order topology \cite{key4} induced by $\le$ coincides with the topology $\mathcal{T}$.

Finally, we recall that a subset $S$ of a topologically orderable gyrogroup $G$ is a \emph{suitable set} \cite{key12} for $G$ if $S$ is discrete, $S\cup \{1\}$ is closed in $G$ and the subgyrogroup generated by $S$ is dense in $G$.

\bigskip
\section{Some properties of strongly topologically orderable gyrogroups}
In this section, we mainly prove that every strongly topologically orderable gyrogroup is either metrizable or has a totally ordered local base at the identity element; moreover, we prove that every strongly topologically orderable gyrogroup is hereditarily paracompact. Indeed, the proofs of some results (such as, Lemma~\ref{1}, Lemma~\ref{2}, Theorem~\ref{5}-Theorem~\ref{6}) in this section are similar to the proofs in \cite{key4}, where the main differences in proofs lie in the use of operations of gyrogroups and the applying of $L$-subgyrogroups. However, we give out the proofs for a convenience of the readers.

First, we give some technical lemmas. The concepts of  {\it isolated from above}, {\it isolated from below}, {\it cofinality $\alpha$ from above}  and {\it cofinality $\alpha$ from below} can be seen in \cite{key4}.

\begin{lemma}\label{1}	
  Let $G$ be a topologically orderable gyrogroup. If the identity element $1$ of $G$ is neither isolated from above nor isolated from below, then its cofinality from above is equal to its cofinality from below.
\end{lemma}

\begin{proof}
 Let $\kappa$ be the cofinality of $e$ from above, and suppose the cofinality from below of $1$ is strictly greater. Let $\{U_\tau:\tau<\kappa\}$ be the family of neighborhoods of $1$ satisfying the following conditions:
	\begin{enumerate}
		\item[(1)] For each $\tau<\kappa$, $U_\tau=(c_{\tau}, d_{\tau})$ for some $c_{\tau}, d_{\tau}\in G$;

\smallskip
		\item[(2)] $\ominus U_{\tau+1}\oplus U_{\tau+1}\subset U_\tau$ for any $\tau<\kappa$;

\smallskip
		\item[(3)] $\lim\limits d_{\tau} = 1$.
	\end{enumerate}
Let $U=\bigcap_{\tau<\kappa}U_{\tau}$. Then $U$ contains an interval $[a, 1]$ where $a{\ne} 1$. Take any $h\in(a, 1)$. Since $d_{\tau}{\notin}U_{\tau}$, it follows that $h{\oplus}d_{\tau}{\notin}U_{\tau+1}$. Otherwise, by applying the left cancellation, we have $$d_{\tau}={\ominus}h\oplus (h{\oplus}d_{\tau})\in{\ominus}h\oplus U_{\tau+1}\subset\ominus U_{\tau+1}\oplus U_{\tau+1}\subset U_{\tau}, $$which leads to a contradiction. Therefore, $h\oplus d_{\tau}\notin[a, 1]$ for all $\tau$. Because $(a, 1)$ is an open neighborhood of $h$, we conclude that $h$ is not in the closure of $\{h\oplus d_{\tau}:\tau<\kappa\}$, which is a contradiction with condition (3).
\end{proof}

\begin{lemma}\label{2}
	Let $G$ be a topologically orderable gyrogroup. Then there exists a totally ordered neighborhood base at the identity element $1$ of $G$.
\end{lemma}

\begin{proof}
	If $1$ is isolated from above or below, the conclusion is obvious. For the case where $1$ is neither isolated from above nor below, the result follows immediately from Lemma~\ref{1}.
\end{proof}

\begin{lemma}\label{3}
Let  $G$ be a strongly topologically orderable gyrogroup.
Then there exists a totally ordered local base $\mathcal{V}$ for the symmetric neighborhoods of the identity element $1$ of $G$ such that $gyr[x,y](W)=W$ for each $x,y\in G$ and $W\in \mathcal{V}$.
\end{lemma}

\begin{proof}
By Lemma~\ref{2}, suppose $\mathcal{U}=\{U_{\alpha}:\alpha<\tau\}$ is a totally ordered neighborhood base at $1$ of $G$. Moreover, since $G$ is a strongly topological gyrogroup, there exists a symmetric neighborhood base $\mathcal{W}=\{W_{\beta}:\beta<\tau\}$ at $1$ such that $gyr[x, y](W)=W$ for each $x,y\in G$ and $W\in \mathcal{W}$. Obviously, we may assume that $\tau$ is a regular cardinal. Then since $\tau$ is a regular cardinal,
there exist a subfamily $\{U_{\alpha_{\sigma}}:\sigma<\tau\}$ of $\mathcal{U}$
and a subfamily $\{W_{\alpha_{\sigma}}:\sigma<\tau\}$ of $\mathcal{W}$ such that:
\begin{enumerate}
	\item[(1)] If $\sigma$ is a successor ordinal, then ${U_{\alpha_{\sigma}}\subseteq W_{\alpha_{\beta}}\subseteq U_{\alpha_{\beta}}}$, where $\sigma=\beta+1$;
    \item[(2)] If $\sigma$ is a limit ordinal, then $U_{\alpha_{\sigma}}\subseteq \bigcup_{\delta<\sigma}W_{\alpha_{\delta}}$.
\end{enumerate}

   Thus, $\mathcal{V}=\{W_{\alpha_{\sigma}}:\sigma<\tau\}$ is a totally ordered neighborhood base at $1$ and satisfies $gyr[x,y](V)=V$ for each $x,y\in G$ and $V\in \mathcal{V}$.
\end{proof}

Now we can prove one of main theorems in this section.

\begin{theorem}\label{4}
Let $G$ be a strongly topologically orderable gyrogroup. Then
\begin{enumerate}
	\item[(1)] $G$ is metrizable or

\smallskip
	\item[(2)] $G$ has a totally ordered local base at the identity element consisting of clopen $L$-subgyrogroup $\mathcal{H}$ such that $gyr[x, y](H)=H$ for any $x, y\in G$ and each $H\in\mathcal{H}$.
\end{enumerate}
\end{theorem}

\begin{proof}
 By Lemma~\ref{3}, there is a totally ordered base $\mathcal{U}=\{U_{\alpha}: \alpha<\tau\}$ consisting of the symmetric neighborhoods of the identity element $1$ of $G$ such that $\forall\ x, y\in G$ and $\forall\alpha<\tau$, $gyr[x, y](U_{\alpha})=U_{\alpha}$. Obviously, we may assume that $\tau$ is a regular cardinal.
     If $\tau<\omega_{1}$, then $G$ is first-countable, hence it follows from \cite[Theorem 2.3]{key5} that $G$ is metrizable. Now assume that $\tau\geq\omega_{1}$. Then it is obvious that we have the following fact.

 \smallskip
 {\bf Fact 1}: For each open neighborhood $U_{\alpha}$ of the identity element, there exists a countable subfamily $\{U_{\alpha_{n}}: n\in\omega\}$ of $\mathcal{U}$ such that $U_{\alpha_{n+1}}\oplus U_{\alpha_{n+1}}\subset U_{\alpha_{n}}\subset U_{\alpha}$ and $\bigcap_{n\in\omega}U_{\alpha_{n}}$ is open in $G$.

By Fact 1, \cite[Proposition 2.11]{key6} and \cite[Proposition 6]{key3}, there exists a local base $\mathcal{H}=\{H_{\alpha}: \alpha<\tau\}$ at the identity element such that the following conditions hold:

 \smallskip
 (i) $H_{\beta}\subset H_{\alpha}$ for any $\alpha<\beta<\tau$;

 \smallskip
 (ii) $H_{\alpha}$ is an open and closed subgyrogroup for each $\alpha<\tau$;

 \smallskip
 (iii) $gyr[x, y](H_{\alpha})=H_{\alpha}$ for any $x, y\in G$ and each $\alpha<\tau$.

 \smallskip
Therefore, it follows from (i)-(iii) that $\mathcal{H}=\{H_{\alpha}: \alpha<\tau\}$ is a totally ordered local base at $1$ consisting of clopen $L$-subgyrogroup.
\end{proof}

The following theorem gives a characterization of a non-metrizable strongly topological gyrogroup which is topologically orderable.

\begin{theorem}\label{5}
Let $G$ be a strongly topological gyrogroup which is not metrizable. Then the following statements are equivalent:
\begin{enumerate}
	\item[(1)] $G$ is topologically orderable.

\smallskip
	\item[(2)] There exists a totally ordered local base at the identity element $1$ of $G$.
	
\smallskip\item[(3)] There exists a totally ordered local base at the identity element $1$ of $G$ consisting of clopen $L$-subgyrogroups.
	
\smallskip\item[(4)]There exists a base $\mathcal{B}$ for $\mathcal{T}$ such that $\mathcal{B}=\bigcup\mathcal{V}$, where $\mathcal{V}=\{\mathcal{V_{\tau}}:\tau<\alpha\}$ is a family of partitions of $G$ into clopen sets, such that $\mathcal{V_{\tau}}$ refines $\mathcal{V_{\sigma}}$ for any $\tau>\sigma$.
\end{enumerate}
\end{theorem}

\begin{proof}
Clearly, we have (1) $\Rightarrow$ (2) and (2) $\Rightarrow$ (3) by Theorem~\ref{4} and Lemma~\ref{3} respectively. Moreover, it is obvious that (4) $\Rightarrow$ (2). Now it suffice to prove that (3) $\Rightarrow$ (4) and (3) $\Rightarrow$ (1).

(3) $\Rightarrow$ (4). Let $\{H_{\alpha}: \alpha<\tau\}$ be a totally ordered local base at the identity element $1$ of $G$ consisting of clopen $L$-subgyrogroups. Since the left translation is a homeomorphism and each $H_{\alpha}$ is an $L$-subgyrogroup, it follows from Theorem~\ref{ttttt} that each family of the left cosets of $H_{\alpha}$ in $G$ form a well ordered
family of partitions of $G$ into clopen sets which together constitute
a base for $G$.

(3) $\Rightarrow$ (1). Let $\{H_{\alpha}: \alpha<\tau\}$ be a totally ordered local base at the identity element $1$ of $G$ consisting of clopen $L$-subgyrogroups. Here we may assume that $H_{\beta}=\bigcap_{\alpha<\beta}H_{\alpha}$ whenever $\beta$ is a limit ordinal, and that $H_{\beta+1}$ is of infinite index in $H_{\beta}$ for each $\beta<\tau$. Then we can define a linear order $\leq$ on $G$ by a complete similar method by \cite[Theorem 6]{key4} such that $\leq$ is compatible with the topology, where such ordering was called {\it integral} in \cite{key4}.
\end{proof}

The following theorem shows that each strongly topologically orderable gyrogroup is hereditarily paracompact.
		
\begin{theorem}\label{7}
	Let $(G, \mathcal{T})$ be a strongly topologically orderable gyrogroup, then $G$ is hereditarily paracompact.
\end{theorem}

\begin{proof}
	If $G$ is metrizable, then it is obvious that $G$ is hereditarily paracompact. Otherwise, by (4) of Theorem~\ref{5}, there exists a family $\mathcal{V}=\{\mathcal{V_{\tau}}:\tau<\alpha\}$ of partitions of $G$ into clopen sets such that the following conditions hold:

\smallskip
(1) $\mathcal{B}=\bigcup\mathcal{V}$ is a base for $\mathcal{T}$.

 \smallskip
(2) For any $\tau>\sigma$, $\mathcal{V_{\tau}}$ refines $\mathcal{V_{\sigma}}$.

Then take any $U, V\in \mathcal{B}$. If $U\cap V\neq\emptyset$, then we have $U\subset V$ or $V\subset U$, hence it follows from \cite[Theorem 4]{key7} that $G$ is hereditarily paracompact.
\end{proof}

Finally, the following theorem gives a characterization of a totally disconnected strongly topological gyrogroup which is topologically orderable.

\begin{theorem}\label{6}
	Let $G$ be a strongly topological gyrogroup. Then the following statements are equivalent:
	\begin{enumerate}
		\item[(1)] $G$ is topologically orderable and totally disconnected.
		
\smallskip\item[(2)] $G$ is topologically orderable and dim $G=0$.
		
\smallskip\item[(3)] There exists a base $\mathcal{B}$ for $\mathcal{T}$ which is union of a well-ordered family $\mathcal{V}=\{\mathcal{V_{\tau}}:\tau<\alpha\}$ of partitions of $G$ into clopen sets such that $\mathcal{V_{\tau}}$ refines $\mathcal{V_{\sigma}}$ for any $\tau>\sigma$.
	\end{enumerate}
\end{theorem}

\begin{proof}
We divide the proof into the following two cases .

\smallskip
{\bf Case 1} $G$ is not metrizable.

By Theorem~\ref{5}, $G$ is topologically orderable, hence it is hereditarily paracompact. Then it suffice to prove that (3) implies (1) and (2). Indeed, it is obvious that dim $G=0$ since zero-dimension is equivalent to dim $G=0$ in a paracompact space. Moreover, by Theorem~\ref{11} in Section 4, if $G$ is not totally disconnected, then it is metrizble, which is contradiction.

\smallskip
{\bf Case 2} $G$ is metrizable.

The implication of (1) $\Rightarrow$ (2) was proved in \cite{HH1965}. It suffice to prove (2) $\Rightarrow$ (3) and (3) $\Rightarrow$ (1).

(2) $\Rightarrow$ (3). By \cite{HH1965}, there exists a non-Archimedean metric $d$ on $G$ inducing the topology of $G$, where the metric $d$ satisfies that, for any $x, z\in G$, we have $d(x, z)\leq \max\{d(x, y), d(y, z)\}$ for all $y$. Let $\mathcal{V}_{n}$ be the partition of $G$ into the balls of the form $B(y, \frac{1}{n})=\{x: d(x, y)<\frac{1}{n}\}$. For any fixed $n\in\mathbb{N}$, these balls are disjoint, hence (3) holds.

(3) $\Rightarrow$ (1). By \cite{HH1974} or \cite[Theorem 1]{key4}, it follows that $G$ has covering dimension zero, then from \cite{HH1965} it follows that each metric space of covering dimension zero is orderable and totally disconnected.
\end{proof}

\smallskip
\section{suitable sets in not totally disconnected or locally compact strongly topologically orderable gyrogroups}
In this section, we mainly prove that every  not totally disconnected or locally compact strongly topologically orderable gyrogroup has a suitable set. First, we give some technical lemmas.

\begin{lemma}\label{8}
		Let $G$ be a topological gyrogroup and $H$ be a $L$-subgyrogroup of $G$. Consider the set $G/H$ and give it the quotient topology for the map $P:G\to G/H$ defined by $P(x)=x\oplus H$. Then $P$ is an open map.
	\end{lemma}
\begin{proof}
Suppose $U$ be an open subset of $G$. Since $P^{-1}(P(U))=U\oplus H$ and $G$ is a topological gyrogroup, it follows that $P(U)$ is open in $G/H$. Thus $P$ is an open map.
\end{proof}

\begin{proposition}\label{9}
	Let $G$ be a topologically orderable gyrogroup and $H$ be a connected $L$-subgyrogroup with at least two distinct elements. Then $H$ must be an open $L$-subgyrogroup of $G$.
\end{proposition}
\begin{proof}
Since $L$ is an $L$-subgyrogroup, let $R$ be the equivalence relation defined by $xRy$ if $y\in x\oplus H$. Now we can endow the set $G/R$ with the quotient topology for the map $\pi: G\rightarrow G/R$ defined by $\pi(x)=x\oplus H$ for each $x\in G$. Then $\pi$ is an open and continuous mapping by Lemma~\ref{8}. Since $H$ is connected, it follows that each $x\oplus H$ is connected and has at least two elements. Hence $\pi^{-1}(z)$ is connected and has at least two elements for each $z\in G/R$, then each $\pi^{-1}(z)$ contains a non-empty open interval since $G$ is topologically orderable. Therefore, $H$ must be an open $L$-subgyrogroup of $G$.
\end{proof}

	\begin{lemma}\label{10}
	The component of a topological gyrogroup is $L$-subgyrogroup.
\end{lemma}

\begin{proof}
	Let $G$ be a topological gyrogroup and $H$ be the component at the identity element $1$. It is obvious that $H$ is a subgyrogroup of $G$. Next, we prove that $H$ is an $L$-subgyrogroup. Indeed, take any $g\in G, x\in H$. Since $gyr[g, x](1)=1$ and $gyr[g, x](H)$ is connected, it follows that $gyr[g, x](H)\subset H$, which shows $gyr[g,x](H)=H$ by \cite[Proposition 6]{key3}. Hence $H$ is an $L$-subgyrogroup of $G$.
\end{proof}

Now we can prove one of the main result in this section.

\begin{theorem}\label{11}
	Let $G$ be a strongly topologically orderable gyrogroup which is not totally disconnected. Then $G$ contains an open and first-countable $L$-subgyrogroup. Thus, $G$ is metrizable.
\end{theorem}

\begin{proof}
	Suppose $H$ is the component at the identity element $1$. Since $G$ is not totally disconnected, it follows from Lemma~\ref{10} that $H$ is a connected $L$-subgyrogroup with at least two elements. Moreover,  Proposition~\ref{9} implies that $H$ is an open $L$-subgyrogroup of $G$.

By \cite[Corollary 1.4]{key8}, $H$ as a topological subspace of $G$ is also a topologically ordered space. Since every topologically ordered space is Hausdorff, it follows that $H$ is Hausdorff. Then $H$ is locally compact by \cite{key11}. Suppose $U$ is compact neighborhood of identity element of $H$. We claim that $H$ is first-countable.

Suppose not, there is a totally ordered base $\mathcal{U}$ of symmetric neighborhoods of the identity element of $H$ from Lemma~\ref{3}. Clearly, there exists a countable subfamily $\{U_{i}: i\in\omega\}$ of $\mathcal{U}$ such that $U_{n+1}\oplus U_{n+1}\subset U_{n}\subset U$ for each $n\in\omega$. Put $V=\bigcap_{n\in\omega}U_n$; then $V$ is a neighborhood of the identity element. Obviously, for any $x,y\in V$, we have $x\oplus y\in U_{n}, \ominus x\in\ominus U_{n}=U_{n}$ for all $n\in\mathbb{N}$, hence we have $x\oplus y\in V, \ominus x\in V$. Thus $V$ is an open subgyrogroup of $H$. By \cite[Proposition 2.11]{key6}, $V$ is a clopen subgyrogroup, so $V$ is compact since $V\subset U$. Since $H$ is connected, it follows that $H=V$. Then $H$ is compact. However, from \cite[Proposition 1.6]{key8}, it follows that each compact connected homogeneous space with at least two elements is not orderable, which is a contradiction. Therefore, $H$ is first-countable. From \cite[Theorem2.3]{key5}, it follows that $H$ is metrizable.

Since $H$ is open and first-countable, it follows that $G$ is first-countable. By applying \cite[Theorem 2.3]{key5} again, $G$ is metrizable.
\end{proof}

By \cite[Theorem 1]{key13} and Theorem~\ref{11}, we have the following important corollary, which gives a generalization of Theorem~\ref{t1}.

\begin{corollary}\label{12}
	Every strongly topologically orderable gyrogroup that is not totally disconnected admits a suitable set.
\end{corollary}

The following theorem shows that each strongly topologically orderable gyrogroup with countable pseudocharacter has a suitable set.

\begin{theorem}\label{14}
Suppose that $G$ is a strongly topologically orderable gyrogroup and $\{1\}$ is a $G_{\delta}$-set, then $G$ is metrizable and has a suitable set.
\end{theorem}

\begin{proof}
	By \cite[Proposition 1.10]{key8}, we conclude that $G$ is first-countable, then $G$ is metrizable by \cite[Theorem 2.3]{key5}. Then it follows from \cite[Theorem 1]{key13} that $G$ contains a suitable set.
\end{proof}

The following theorem gives a characterization of a separable and totally disconnected topological gyrogroup such that it is topologically orderable.

\begin{theorem}\label{15}
	Let $G$ be a separable and totally disconnected topological gyrogroup. Then $G$ is topologically orderable space if and only if it is metrizable and zero-dimensional.
\end{theorem}

\begin{proof}
Let $G$ be a topologically orderable gyrogroup. Then it follows from \cite[Proposition 1.8]{key8} that $G$ is first-countable, then it is metrizable by \cite[Theorem 2.3]{key5}. Since a metrizable totally disconnected space is zero-dimensional by \cite{HH1965}, we conclude that $G$ is zero-dimensional. Conversely, each separable
metric zero-dimensional space is orderable has been proved in \cite{L1963}.
\end{proof}

Finally, we prove the second main theorem in this section. First, we give some technical lemmas.

\begin{lemma}\label{16}
Let $G$ be a totally disconnected locally compact strongly topological gyrgroup. Then every neighbourhood of the identity contains a compact open $L$-subgyrgroup.
\end{lemma}

\begin{proof}
Let $G$ have a symmetric neighborhood base $\mathcal{B}$ at the identity element such that $gyr[x, y](B)=B$ for any $x, y\in G$ and $B\in\mathcal{B}$. Since $G$ is totally disconnected and locally compact, it follows from Vedenissov's Theorem that $G$ is zero-dimension, hence there exists a compact open neighborhood $U$ of $1$. For each $x\in U$, there exist $U_{x}, V_{x}\in\mathcal{B}$ such that $x\oplus U_{x}\subset U$,  $U_{x}\oplus x\subset U$ and $V_{x}\oplus V_{x}\subset U_{x}$. By the compactness of $U$, there exists a finite set $\{x_{1}, \ldots, x_{n}\}$ such that $U\subset (\bigcup_{i=1}^{n}(x_{i}\oplus V_{x_{i}}))\cap (\bigcup_{i=1}^{n}(V_{x_{i}}\oplus x_{i}))$. Put $V=\bigcap_{i=1}^{n}V_{x_{i}}$. Then
\begin{align*}
	    		U\oplus V&\subset (\bigcup_{i=1}^{n}x_{i}\oplus V_{x_{i}})\oplus V\\
	    		&\subset\bigcup_{i=1}^{n}((x_{i}\oplus V_{x_{i}})\oplus V_{x_{i}})\\
	    		&=\bigcup_{i=1}^{n}((x_{i}\oplus (V_{x_{i}}\oplus gyr[x_{i}, V_{x_{i}}](V_{x_{i}})))\\
	    		&=\bigcup_{i=1}^{n}((x_{i}\oplus (V_{x_{i}}\oplus V_{x_{i}}))\\
	    		&=\bigcup_{i=1}^{n}((x_{i}\oplus U_{x_{i}})\\
	    		&\subset U.
	    	\end{align*}

Therefore, $V\subset U\oplus V\subset U$, which implies that $V\oplus V\subset (U\oplus V)\oplus V\subset U\oplus V\subset U$ and $(V\oplus V)\oplus V\subset U$. Moreover, since $gyr[x, y](V)=V$ for any $x, y\in G$, it follows that $V\oplus (V\oplus V)=(V\oplus V)\oplus V$. Therefore, the subgyrgroup $H$ generated by $V$ is contained in $U$ and open in $G$, thus it is closed and compact. Because $gyr[x, y](V)=V$ for any $x, y\in G$, we conclude that $H$ is an $L$-subgyrgroup.
\end{proof}

\begin{lemma}\label{17}
Let $G$ be an infinite, locally compact, totally disconnected strongly topologically orderable gyrogroup. Then either $G$ is discrete or $G$ contains a clopen $L$-subgyrgroup $H$ which as a topological space is homeomorphic with the Cantor set with its usual topology.
\end{lemma}

\begin{proof}
  By Lemma~\ref{16}, let $H$ be a compact, clopen $L$-subgyrgroup. We conclude that $H$ is metrizable, thus it is first-countable. Otherwise, it follows from Theorem~\ref{4} that $H$ has a totally ordered local base $\{U_{\alpha}: \alpha\in\tau\}$ at the identity element consisting of clopen $L$-subgyrogroup, where $\tau\geq\omega_{1}$ and $U_{\beta}\setminus U_{\alpha}\neq\emptyset$ for any $\alpha>\beta$. Since $H$ is compact, the left cosets of each $U_{\alpha}$ is finite in number, which is denoted by $n_{\alpha}$. Then $n_{\alpha}>n_{\beta}$ for any $\alpha>\beta$ since $U_{\beta}\setminus U_{\alpha}\neq\emptyset$. However, the set $\{n_{\alpha}: \alpha\in\tau\}$ is countable, hence there exists $\gamma<\omega_{1}$ such that $n_{\alpha}=n_{\gamma}$ for any $\alpha>\gamma$. Thus $H$ must be first-countable, which is a contradiction. Therefore, $H$ is metrizable.

If $G$ is not discrete, then $H$ is also not discrete. Since $H$ is compact, metrizable and totally disconnected, it follows that $H$ is homeomorphic to the cantor set.
\end{proof}

Now we can prove the following theorem.

\begin{theorem}
Each locally compact strongly topologically orderable gyrogroup is metrizable; thus it has a suitable set.
\end{theorem}

\begin{proof}
By Theorem~\ref{11} and Lemma~\ref{17}, $G$ is metrizable. Since each metrizable strongly topological gyrogroup has a suitable set \cite[Theorem 1]{key13}, we conclude $G$ has suitable set.
\end{proof}

\section{dense subgyrogroups of stronglly topologically orderable gyrogroups}
In this section, we mainly show that a strongly topologically orderable gyrogroup has a (closed) suitable set if and only if each dense subgyrogroup of it has a (closed) suitable set. First, we need some lemmas.

\begin{lemma}\label{18}
Let $G$ be a strongly topologically orderable gyrogroup and $D$ be a discrete subset of $G$. Then the points of $D$ can be separated by pairwise disjoint neighborhoods.
\end{lemma}	
\begin{proof}
	 If $G$ is metrizable, then the conclusion is clearly valid. If $G$ is non-metrizable, then let $D=\left \{ x_{i} :i\in I \right \}$. Since $G$ is a strongly topologically orderable gyrogroup, it follows from Theorem~\ref{4} that there exists a local base $\left \{ H_{\alpha }:\alpha < \tau \right \} $ at 1 of $G$ consisting of clopen $L$-subgyrogroups of $G$ such that $H_{\beta } \subseteq H_{\alpha } $ for $\alpha < \beta < \tau $ and $gyr[x,y](H_{\alpha})=H_{\alpha}$ for any $x,y \in G,\alpha<\tau$. For any $i\in I$, since $D$ is discrete, there exists $\alpha \left ( i \right ) < \tau $ such that $(x_{i} \oplus H_{\alpha \left ( i \right ) })\cap D=\left \{ x_{i} \right \}. $
		
	 Let us prove that the family $\{x_{i} \oplus H_{\alpha \left ( i \right ) }:i\in I\}$ is disjoint. Suppose not, then there exist distinct two elements $i, j\in I$ such that $(x_{i} \oplus H_{\alpha \left ( i \right )})\cap(x_{j} \oplus H_{\alpha \left ( j \right )})\neq \emptyset $. Without loss of generality, we may assume that $\alpha \left ( i \right )\leq \alpha \left ( j \right ) $, then $H_{\alpha \left ( j \right ) } \subseteq H_{\alpha \left ( i \right ) } $. Hence there exist $g_{i}\in H_{\alpha \left ( i \right ) } $, $g_{j}\in H_{\alpha \left ( j \right ) }$ such that $x_{i}\oplus g_{i} = x_{j}\oplus g_{j}$. Therefore,
		\begin{align*}
			x_{j}&= x_{j}\oplus (g_{j}\oplus(\ominus g_{j}))\\
			&=(x_{j}\oplus g_{j})\oplus gyr[x_{j}, g_{j}](\ominus g_{j})\\
			&\subseteq(x_{j}\oplus g_{j})\oplus gyr[x_{j}, g_{j}](H_{\alpha \left ( j \right ) })\\
			&\subseteq(x_{i}\oplus H_{\alpha \left ( i \right ) })\oplus H_{\alpha \left ( j \right ) }\\
			&=x_{i}\oplus (H_{\alpha \left ( i \right ) }\oplus H_{\alpha \left ( j \right ) })\\
			&\subseteq x_{i}\oplus (H_{\alpha \left ( i \right ) }\oplus H_{\alpha \left ( i \right ) })\\
			&\subseteq x_{i}\oplus H_{\alpha \left ( i \right ) }.
		\end{align*}		
This leads to a contradiction since $(x_{i} \oplus H_{\alpha \left ( i \right ) })\cap D=\left \{ x_{i} \right \}.$	Therefore, the family $\{x_{i} \oplus H_{\alpha \left ( i \right ) }:i\in I\}$ is disjoint, then $D$ can be separated by pairwise disjoint neighborhoods.
        \end{proof}
	
		\begin{lemma}\label{19}
Let $\tau$ be an infinite cardinal. Suppose that $G$ is a strongly topologically orderable gyrogroup, and suppose that $\{H_{\alpha}: \alpha<\tau\}$ is a base at the identity of $G$ consisting of clopen $L$-subgyrogroups satisfies the following conditions:

\smallskip
(1) $H_{\beta}\subset H_{\alpha}$ for any $\alpha<\beta<\tau$;

 \smallskip
 (2) $gyr[x, y](H_{\alpha})=H_{\alpha}$ for any $x, y\in G, \alpha<\tau$.

\smallskip
If $D$ is a subset of $G$ and $f:D\rightarrow \tau $ is a function such that the family $\gamma =\left \{ x\oplus H_{f\left ( x \right ) } :x\in D \right \} $ is disjoint, then $y\in G$ is an accumulation point of the family $\gamma$ if and only if $y$ is an accumulation point of $D$.
    	\end{lemma}	
	    \begin{proof}
	    	The necessity is obvious. Suppose that $y\in G$ is an accumulation point of $\gamma$, then $y\notin \cup \gamma$. We conclude that the following claim holds:

\smallskip	    	
{\bf \mbox{Claim:}} If $x \in D, \alpha < \tau\ and\ \left ( y\oplus H_{\alpha }  \right )\cap\left ( x\oplus H_{f\left (  x\right )  }  \right ) \neq \emptyset,\ then\ \alpha < f\left ( x \right ).$

\smallskip	 	
Assume the contrary that, for some $ x \in D$ and $\left ( y\oplus H_{\alpha }  \right )\cap\left ( x\oplus H_{f\left (  x\right )  }  \right ) \neq \emptyset$ such that $\alpha \geq f\left ( x \right ) $, which shows that $H_{\alpha}\subset H_{f(x)}$. Then there exist $g_{\alpha}\in H_{\alpha } $, $g_{f\left ( x \right )}\in H_{f\left ( x \right ) }$ such that $y\oplus g_{\alpha} = x\oplus g_{f\left ( x \right )}$. Therefore, we have
	    	\begin{align*}
	    		y&= y\oplus (g_{\alpha}\oplus(\ominus g_{\alpha}))\\
	    		&=(y\oplus g_{\alpha})\oplus gyr[y, g_{\alpha}](\ominus g_{\alpha})\\
	    		&\subseteq(x\oplus g_{f\left ( x \right )})\oplus gyr[y, g_{\alpha}](H_{\alpha})\\
	    		&\subseteq(x\oplus H_{f\left ( x \right )})\oplus H_{\alpha}\\
	    		&=x\oplus (H_{f\left ( x \right ) }\oplus H_{\alpha })\\
	    		&\subseteq x\oplus (H_{f\left ( x \right ) }\oplus H_{f\left ( x \right ) })\\
	    		&= x\oplus H_{f\left ( x \right ) }.
	    	\end{align*}
This is a contradiction with $y\notin \bigcup \gamma$.
	    	
Let $U$ be an arbitrary neighborhood of $y$ in $G$. Then there exists $\alpha < \tau$ such that $y\oplus H_{\alpha }\subseteq U $. Since  $y\in\overline{ \cup \gamma}$, there exists $x\in D$ such that $\left ( y\oplus H_{\alpha }  \right )\cap\left ( x\oplus H_{f\left (  x\right )  }  \right ) \neq \emptyset$, hence $\alpha<f(x)$ by Claim. Then
since $H_{f(x)}\subset H_{\alpha}$ and $H_{\alpha}$ is an $L$-subgyrogroup, it follows that $x\in y\oplus H_{\alpha }$. Therefore, $x\in \left ( y\oplus H_{\alpha } \right )\cap D\subseteq U\cap D\neq \emptyset $, so $y\in \overline{ D}$. This proof has been completed.
	    \end{proof}
	
		\begin{lemma}\label{20}
		Suppose that  $G$ be a strongly topologically orderable gyrogroup and $D$ be a discrete and closed subset of $G$ (resp. has at most one accumulation point). If $L$ is a subgyrogroup of $G$ such that $D\subseteq \overline {L}$, then there exists a subset $F\subseteq L$ satisfy the following conditions:

\smallskip
(i) $F$ is discrete in itself;

\smallskip
(ii) $F$ is  closed in $L$ (resp. has at most one accumulation point in $L$);

\smallskip
(iii) $D\subseteq \overline {F}$ in $G$.
	    \end{lemma}	
	    \begin{proof}
	    If $G$ is metrizable, the conclusion is immediate. Suppose $G$ is non-metrizable topologically orderable gyrogroup, it follows from Theorem~\ref{5} that there is a decreasingly well-ordered base $\left \{ H_{\alpha }:\alpha < \tau \right \} $ at 1 for some uncountable regular cardinal $\tau$. It is easy to see that every subset of $G$ of cardinality less than $\tau$ is closed and discrete in $G$. For each $\alpha < \tau$, let $V_{\alpha }=\overline {L} \cap H_{\alpha }.$ Hence the family $\{V_{\alpha }:\alpha < \tau\}$ is a decreasing base at 1 of $\overline {L}$. Then $\overline {L}$ is also a topologically orderable gyrogroup. Since discrete subset $D\subseteq \overline {L}$, by Lemma~\ref{18}, there exists a function $f:D\rightarrow \tau $ such that the family $\gamma =\left \{ x\oplus V_{f\left ( x \right ) } :x\in D \right \} $ is pairwise disjoint. In $G$, the set $D$ can have only one accumulation point, then we may assume the accumulation point is the identity element 1, thus $1\not\in D$. Then by Lemma~\ref{19}, the identity element 1 is the unique accumulation point of $\gamma$ in $G$. Moreover, if $D$ is closed in $G$, then the family $\gamma$ will be discrete $G$.

For any $x\in D$, define a closed discrete subset $F_{x} $ of $L\cap \left (x\oplus V_{f\left ( x \right ) } \right ) $ as follows:

\smallskip
 (1) If $x\in L$, $F_{x}=\left \{ x \right \} $.

\smallskip
 (2) If $x\notin L$, then for each $f(x)\leq \alpha\leq\tau$, pick $z_{x, \alpha}\in (L\cap (x\oplus V_{\alpha}))\setminus\{1\}$; then put $F_{x}=\left \{z_{x, \alpha}:f\left ( x \right ) \leq \alpha < \tau   \right \}$.
	
Now set $F=\bigcup _{x\in D} F_{x}$. Clearly, $1\not\in F$. Finally, it suffices to claim that $F$ is discrete in itself and has at most one accumulation point in $L$.

Indeed, fix any $x\in D$. According to the definition of $F_{x}$, if $x\in L$, then $F_{x}=\{x\}$ is a discrete subset. Otherwise, $x$ is the limit point of $\left \{z_{x, \alpha}:f\left ( x \right ) \leq \alpha < \tau   \right \}\subset L$ by our choice of points, then the point $x$ is the unique accumulation point of $F_{x}$ in $G$ and $F_{x}$ is closed in $L$ since $x\not\in L$. Since the family $\gamma$ has at most one accumulation point $1$ in $G$ and $F_{x}\subseteq x\oplus V_{f\left ( x \right ) }$ for each $x\in D$, it follow that $F\cup\{1\}$ is closed in $L$ and the set $F$ is discrete in $L\setminus\{1\}$. Hence $F$ is discrete in itself and has at most one accumulation point in $L$. Clearly, if $D$ is closed in $G$, then $F$ is closed in $L$.
	    \end{proof}

By Lemma~\ref{20}, we can prove the following main theorem.
	
\begin{theorem}\label{21}
Let  $G$ be a strongly topologically orderable gyrogroup and $H$ be a dense subgyrogroup of $G$. If $G$ has a (closed) suitable set, then $H$ also has a (closed) suitable set.
\end{theorem}

\begin{proof}
Let $D$ be a (closed) suitable set of $G$ such that $1\not\in S$. Then $S$ is discrete in itself and $D\subset \overline{H}=G$. By Lemma~\ref{20}, there exists a subset $F\subset H$ such that $F$ is discrete in itself, $F$ has at most one accumulation point in $H$ (resp. is closed in $H$) and $D\subseteq \overline {F}$ in $G$. Therefore, $F$ is a (closed) suitable set for $H$.
\end{proof}
     	
 The following questions are still unknown for us.

	    \begin{question}\label{22}
	  Suppose that $G$ is a strongly topologically orderable gyrogroup with a suitable set and $H$ is a subgyrogroup of $G$, does $H$ have a suitable set?
	    \end{question}

	    \begin{question}\label{23}
	  Suppose that $G$ is a strongly topologically orderable gyrogroup with a suitable set and $H$ is a non-closed subgyrogroup of $G$, does $H$ have a closed suitable set?
	    \end{question}

{\bf Acknowledgements}
We wish to express our sincere thank
the reviewers for careful reading preliminary version of this paper and providing many valuable suggestions.

\end{document}